\documentclass[12pt]{amsart}
\usepackage{amsmath}
\usepackage{amsfonts}
\usepackage{amssymb}
\usepackage{latexsym}
\usepackage{times,a4wide}

\newtheorem{theorem}{Theorem}[section]
\newtheorem{lemma}[theorem]{Lemma}

\newtheorem{proposition}[theorem]{Proposition}

\newtheorem{Definition}{Definition}[section]

\newtheorem{otherth}{\bf Theorem}

\newtheorem{otherl}{\bf Lemma}

\newtheorem{rk}{\textbf{Remark}}

\newcommand{\D}{\mathbb D}

\numberwithin{equation}{section}

\begin{document}

\title[Schatten class Toeplitz operators]{Schatten class Toeplitz operators acting on large weighted Bergman spaces}
\

\author[H. Arroussi]{Hicham Arroussi}
\address{Hicham Arroussi \\Departament de Matem\`{a}tica Aplicada i Analisi\\
Universitat de Barcelona\\ Gran Via 585 \\
08007 Barcelona\\
Spain} \email{arroussihicham@yahoo.fr}

\author[I. Park]{Inyoung Park}
\address{Inyoung Park \\ Center for Geometry and its Applications \\
Pohang University of Science and Technology \\Pohang 790-784\\ South Korea} \email{iypark26@postech.ac.kr}

\author[J. Pau]{Jordi Pau}
\address{Jordi Pau \\Departament de Matem\`{a}tica Aplicada i Analisi\\
Universitat de Barcelona\\ Gran Via 585 \\
08007 Barcelona\\
Spain} \email{jordi.pau@ub.edu}


\thanks{The second author was supported by the National Research Foundation of Korea (NRF) grant funded by the Korea government (MSIP) (No.2011-0030044), while the third author was
 supported by  DGICYT grants MTM2011-27932-C02-01 and MTM2014-51834-P
(MCyT/MEC) and 2014SGR289 (Generalitat de Catalunya)}

\maketitle


\begin{abstract}
 A full description of the membership in the Schatten ideal $S_ p(A^2_{\omega})$ for $0<p<\infty$ of the Toeplitz operator acting on large weighted Bergman spaces is  obtained.
\end{abstract}


\section{Introduction}

Let $H(\D)$ denote the space of all analytic functions on $\D,$ where $\D$ is the open unit disk in the complex plane $\mathbb{C}$.
A weight is a positive function $\omega\in L^1(\mathbb{D}, dA),$ with $dA(z) = \frac{dxdy}{\pi}$ being the normalized area measure on $\D$.
The weighted Bergman space $A^2_{\omega}$ is the space of all  functions $f\in H(\D)$ such that
\[
\|f \|_{A^2_{\omega}} =\bigg( \int_{\D}|f(z)|^2  \,\omega(z)\, dA(z)\bigg)^{\frac{1}{2}} < \infty.
\]
We are going to study Toeplitz operators acting on these weighted Bergman spaces, for a certain class $\mathcal{W}$ of radial rapidly decreasing weights. The class $\mathcal{W}$ is the same class of weights considered in \cite{PP1} and \cite{GP1}, and consists of those radial decreasing weights  of the form $\omega(z)=e^{-2\varphi(z)}$,
where $\varphi\in C^2(\D)$ is a radial function such  that
$\left(\Delta\varphi(z)\right)^{-1/2}\asymp \tau(z)$, for
 a radial positive function $\tau(z)$ that decreases to $0$ as
$|z|\rightarrow 1^{-}$, and $\lim_{r\to 1^-}\tau'(r)=0.$ Here
$\Delta$ denotes the standard Laplace operator.  Furthermore, we shall also suppose that either there exist  a constant $C>0$  such that
$\tau(r)(1-r)^{-C}$ increases for $r$ close to $1$ or
 \[
 \lim_{r\to 1^-}\tau'(r)\log\frac{1}{\tau(r)}=0.
 \]
 The prototype is  the exponential type weight
\begin{equation}\label{Exp}
\omega_{\alpha}(z) = \exp\left(\frac{-1}{(1-|z|^2)^\alpha}\right)  ,   \  \alpha > 0 ,
\end{equation}
 but the class $\mathcal{W}$ also contains the double exponential weights
\[\omega(z)=\exp\left(-\gamma\exp\Big(\frac{\beta}{(1-|z|)^\alpha}\Big)\right),\quad
\alpha,\beta,\gamma>0.\]
For the weights $\omega$ in our class,  the point evaluations $L_ z$ are bounded linear functionals on $A^2_{\omega}$ for each $z\in \D$. In particular, the space $A^2_{\omega}$ is a reproducing kernel Hilbert space:  for each $z\in \D$, there are functions $K_ z\in A^2_{\omega}$ with $\|L_ z\|=\|K_ z\|_{A^2_{\omega}}$ such that $L_ z f=f(z)=\langle f, K_ z \rangle _{\omega}$, where
\[
\langle f,g \rangle_{\omega}=\int_{\D} f(z)\,\overline{g(z)} \,\omega(z)\,dA(z)
\]
is the natural inner product in $L^2(\D,\omega dA)$. The function $K_z$ is called the reproducing kernel for the Bergman space $A^2_{\omega}$ and has the property that $K_ z(\xi)=\overline{K_{\xi}(z)}$. The study of the
 properties of the Bergman spaces with exponential type weights has attracted a lot of attention in recent years \cite{D2,GP1,KrMc,LR1,PP1,PP2}, and new techniques different from the ones used on  standard Bergman spaces are required.

Let $\omega \in \mathcal{W}$ and $\mu$ be a finite positive  Borel measure on $\D$.
The Toeplitz operator $T^\omega_\mu$ with symbol $\mu $ is  given  by
\[
 T_\mu f(z) = T^\omega_\mu f (z) :  = \int_{\D}{f(\xi)\,\overline{K_z(\xi)}\,\omega(\xi)\, d\mu(\xi)},\qquad f\in H(\D).
\]

There is a lot of work on the study of Toeplitz operators acting on several spaces of holomorphic functions \cite{CIJ, HL, Lue, Me, RW1, Z1}, and the theory is especially well understood in the case of Hardy spaces or standard Bergman spaces (see \cite{Zhu} and the references therein). Luecking \cite{Lue} was the first to study Toeplitz operators on the Bergman spaces with measures as symbols, and the study of Toeplitz operators acting on large weighted Bergman spaces was initiated by Lin and Rochberg \cite{LR2}, where they proved the following result (as the conditions of the weights are slightly different, we offer a proof in Section 3).

\begin{theorem}\label{TOPQ}
Let  $\omega \in \mathcal{W}$. Then
\begin{enumerate}
\item[(i)] $ T_\mu : A^2_{\omega} \rightarrow A^2_{\omega} $ is bounded if and only if for each $\delta>0$ small enough, one has
 \begin{equation}\label{tauomega1}
C_{\mu}:=\sup_{\substack{z\in\D}}\,\widehat{\mu}_{\delta}(z) < \infty.
\end{equation}
Moreover, in that case,
$
\|T_{\mu}\|\asymp  C_{\mu}.
$
\item[(ii)] $ T_\mu : A^2_{\omega} \rightarrow A^2_{\omega} $ is compact if and only if for each $\delta>0$ small enough, one has
\begin{equation}\label{CompCond}
\lim\limits_{r\rightarrow1^-}\sup\limits_{|z|> r} \widehat{\mu}_{\delta}(z) = 0.
\end{equation}
\end{enumerate}
\end{theorem}
\mbox{}
\\
Here $\widehat{\mu}_{\delta}$ is the averaging function defined as
\[
\widehat{\mu}_{\delta}(z)=\frac{\mu(D(\delta \tau(z))}{\tau(z)^2},\qquad z\in \D.
\]
Here $\tau(z)$ is the function associated to the weight $\omega$, and  $D(\delta \tau(z))$ denotes an euclidian disk centered at $z$ and radius $\delta \tau(z)$.\\

Before introducing our main result, we need to review some definitions and properties of Schatten ideals. Let $H$  be a separable Hilbert space, and  $0<p<\infty$. The Schatten class $S_ p=S_ p(H)$ consists of those compact operators $T$ on $H$
for which its sequence of singular numbers $\{\lambda_ n\}$  belongs to the
sequence space $\ell^p$ (the singular numbers are the square roots of the eigenvalues of the positive operator $T^*T$, where $T^*$ is the Hilbert adjoint of $T$). For $p\ge 1$, the class $S_ p$ is a Banach space
with the norm $\|T\|_ p=\left (\sum_ n |\lambda_ n|^p\right
)^{1/p},$ while for $0<p<1$ one has \cite{M} the inequality $\|S+T\|_
p^p\le \|S\|_ p^p+\|T\|_ p^p.$ Also, one has $T\in S_p$ if and only if $T^*T \in  S_{p/2}.$ We refer to \cite[Chapter 1]{Zhu} for a brief account on Schatten classes.\\

 Concerning membership in Schatten ideals of Toeplitz operators, a description for the standard Bergman spaces was obtained by Luecking \cite{Lue}. For the case of our large Bergman spaces we have the following characterization, that is the main result of the paper, and solves a problem posed by Lin and Rochberg in \cite{LR2}.
\begin{theorem}\label{MT3}
Let $\omega \in \mathcal{W}$, and $\mu$ be a finite positive Borel measure on $\D$. For $0<p<\infty$,  the Toeplitz operator $T_{\mu}$ is in Schatten class $S_ p(A^2_{\omega})$ if and only if the averaging function $\widehat{\mu}_{\delta}$ is in $L^p(\D,d\lambda_{\tau})$, where $d\lambda_{\tau}=\tau(z)^{-2}dA(z)$.
\end{theorem}
\mbox{}
\\

Throughout this work, the letter $C$ will denote an absolute constant whose value may change at different occurrences. We also use the notation
$a\lesssim b$ to indicate that there is a constant $C > 0$ with $a \leq C b,$ and the notation $a \asymp b$ means that $a \lesssim b$ and $b \lesssim a.$\\

The paper is organized as follows. In Section 2 we recall some notation
and preliminary results which will be used later.
 We prove Theorem \ref{TOPQ} on the boundedness and compactness of Toeplitz operators  in Section 3,  and finally in Section 6 we describe the membership of the Toeplitz operator in the Schatten ideals $S_ p(A^2_{\omega})$.

\section{Preliminaries and basic properties}

A positive function $\tau$ on $\D$ is said to be in the class $\mathcal{L}$ if satisfies the following two properties:
\begin{itemize}
\item[(A)] There is a constant $c_ 1$ such that $\tau(z)\le c_ 1\,(1-|z|)$ for all $z\in \D$;\\
\item[(B)] There is a constant $c_ 2$ such that $|\tau(z)-\tau(\zeta)|\le c_ 2\,|z-\zeta|$ for all $z,\zeta\in \D$.\\
\end{itemize}
 We also use the notation
 \[m_\tau : = \displaystyle \frac{\min(1, c_1^{-1}, c_2^{-1})}{4},\]
  where $c_1$ and $c_2$ are the constants appearing in the previous definition. For $a\in \D$ and $\delta>0$, we use  $D(\delta\tau(a))$ to denote the euclidian disc centered at $a$ and radius $\delta \tau(a)$.
It is easy to see from conditions (A) and (B) (see \cite[Lemma 2.1]{PP1}) that if $\tau \in\mathcal{L}$  and $ z \in  D(\delta\tau(a)) ,$ then
\begin{equation}\label{asymptau}
 \frac{1}{2}\,\tau(a)\leq \tau(z) \leq  2 \,\tau(a),
\end{equation}
 for sufficiently small $\delta > 0 ,$ that is, for $\delta \in (0, m_\tau).$ This fact will be used many times in this work.

\begin{Definition}
We say that a weight $\omega$ is in the class  $\mathcal{L}^*$ if it is of the form $\omega =e^{-2\varphi}$, where $\varphi \in C^2(\D)$ with $\Delta \varphi>0$, and $\big (\Delta \varphi (z) \big )^{-1/2}\asymp \tau(z)$, with $\tau(z)$ being a function in the class $\mathcal{L}$. Here $\Delta$ denotes the classical Laplace operator.
\end{Definition}
It is straightforward to see that $\mathcal{W}\subset \mathcal{L}^*$. The following result is from \cite[Lemma 2.2]{PP1} and  gives the boundedness of the point evaluation functionals on $A^2_{\omega}$.

\begin{otherl}\label{subHarmP}
Let $\omega \in \mathcal{L}^*$, $0<p<\infty$, and let $z\in \D$. If $\beta \in \mathbb{R}$ there exists $ M \geq 1$ such that
\[ |f(z)|^p \omega(z)^{\beta} \leq  \frac{M}{\delta^2\tau(z)^2}\int_{D(\delta\tau(z))}{|f(\xi)|^p \omega(\xi)^{\beta} \   dA(\xi)},\]
for all $ f \in H(\mathbb{D})$  and all $\delta > 0 $ sufficiently small.
\end{otherl}

\mbox{}
\\
The following lemma on coverings is due to
Oleinik, see \cite{O}.
\begin{otherl}\label{Lcoverings}
Let $\tau$ be a positive function in $\D$ in the class
$\mathcal{L}$, and let $\delta\in (0,m_{\tau})$. Then there exists
a sequence of points $\{z_ j\}\subset \D$, such that the following
conditions are satisfied:
\begin{enumerate}
\item[$(i)$] \,$z_ j\notin D(\delta \tau(z_ k))$, \,$j\neq k$.

\item[$(ii)$] \, $\bigcup_ j D(\delta \tau(z_ j))=\D$.

\item[$(iii)$] \, $\tilde{D}(\delta \tau(z_ j))\subset
D(3\delta\tau(z_ j))$, where $\tilde{D}(\delta \tau(z_
j))=\bigcup_{z\in D(\delta \tau(z_ j))}D(\delta \tau(z))$,
$j=1,2,\dots$

\item[$(iv)$] \,$\big \{D(3\delta \tau(z_ j))\big \}$ is a
covering of $\D$ of finite multiplicity $N$.
\end{enumerate}
\end{otherl}
The multiplicity $N$ in the previous Lemma  is independent of $\delta$, and it is easy to see that one can take, for example, $N=256$. Any sequence satisfying the conditions in Lemma \ref{Lcoverings} will be called a $(\delta,\tau)$-\textbf{lattice}. \\

\subsection{Reproducing kernels estimates and test functions}\label{sec2:1}
 The next result (see \cite{BDK,LR1,PP1} for (a), and  \cite[Lemma 3.6]{LR2} for part (b)) provides useful estimates involving reproducing kernels.

\begin{otherth}\label{RK-PE}
Let $K_z $ be the reproducing kernel of  $A^2_{\omega}$ where $\omega$ is a weight in  the class $\mathcal{W}$. Then
\begin{enumerate}
\item[(a)] For each $z\in \D$, one has
\begin{equation}\label{Eq-NE}
\|K_z\|_{A^2_{\omega}}^2 \, \omega(z)  \asymp \frac{1}{\tau(z)^2},\qquad z\in \D.
\end{equation}
\item[(b)] For all $\delta \in (0,m_{\tau})$ sufficiently small, one has
\begin{equation}\label{RK-Diag}
|K_ z(\zeta)| \asymp \|K_ z\|_{A^2_{\omega}}\cdot \|K_{\zeta}\|_{A^2_{\omega}} ,\qquad \zeta \in D(\delta \tau(z))
\end{equation}
\end{enumerate}
\end{otherth}

\mbox{}
\\
The following result on test functions obtained in \cite{PP1} is basic for our characterization of Schatten class Toeplitz operators.

\begin{otherl}\label{Borichevlemma}
Let $N\in\mathbb{N}\setminus\{0\}$  and $\omega\in \mathcal{W}$. Then, there is a number $\rho_
0\in (0,1)$ such that for each $a\in \D$ with $|a|\geq \rho_ 0$,
there is a function $F_{a,N}$ analytic in $\D$ with
\begin{equation}\label{BL1}
|F_{a,N}(z)|^2\,\omega(z)\asymp 1 \quad \textrm{ if }\quad
|z-a|<\tau(a),
\end{equation}
and
\begin{equation}\label{BL2}
|F_{a,N}(z)|\,\omega(z)^{1/2}\lesssim \min \left
(1,\frac{\min\big(\tau(a),\tau(z)\big)}{|z-a|} \right )^{3N}, \quad
z\in \D.
\end{equation}
Moreover, the function $F_{a,N}$  belongs to $A^2_{\omega}$ with $$\|F_{
a,N}\|_{A^2_{\omega}}\asymp \tau(a), \quad \rho_ 0 \le |a|<1.$$
\end{otherl}
\mbox{}
\\

\subsection[Carleson type measures]{Carleson type measures}\label{sec1:4}

Let $\mu$ be a finite positive Borel measure on $\D$. We say that $\mu$ is  a Carleson measure for $A^2_{\omega} $ if there exists a finite positive constant  $C$  such that
\[
\int_{\D}{|f(z)|^2\ d\mu(z)}\leq C \|f\|^2_{A^2_{\omega}},
\]
for all $f \in A^2_{\omega}.$
Thus, $\mu$ is Carleson for $A^2_{\omega}$ when the inclusion
$I_\mu :A^2_{\omega}\rightarrow L^2(\D,\mu)$ is bounded.
Next  result was  proved in \cite{PP1}.

\begin{otherth}\label{cdecarl1}
Let $\omega\in\mathcal{W}$  and  $\mu$ be a finite positive Borel measure on $\D$.   Then
\begin{enumerate}
\item [(a)]
$I_\mu :A^2_{\omega}\rightarrow L^2(\D,\mu)$ is bounded if and only if  for each sufficiently small $\delta > 0$ we have
\begin{equation*}\label{uniq}
K_{\mu,\omega} :  =  \displaystyle\sup_{\substack{a\in\mathbb{D}}}\frac{1}{\tau(a)^{2}}\int_{D(\delta\tau(a))}{\omega(\xi)^{-1}  d\mu(\xi)} < \infty.
\end{equation*}
Moreover, in that case,  $K_{\mu,\omega} \asymp \|I_\mu\|^2_{A^2_{\omega}\rightarrow L^2(\D,\mu)}.$
\item[(b)] $I_\mu :A^2_{\omega}\rightarrow L^2(\D,\mu)$ is compact if and only if
 for each sufficiently small $\delta > 0$ we have
\begin{equation*}\label{uni}
 \lim\limits_{r\rightarrow1^-}\sup\limits_{\substack{|a| > r}}\frac{1}{\tau(a)^{2}}\int_{D(\delta\tau(a))}{\omega(\xi)^{-1}   d\mu(\xi)} = 0.
\end{equation*}
\end{enumerate}
\end{otherth}

\mbox{}
\\
\begin{lemma}\label{LF1}
Let $\omega \in \mathcal{W}$, and assume that $\widehat{\mu}_{\delta}$ is in $L^{\infty}(\D)$ for some small $\delta>0$. Then
\[
\langle T_{\mu} f, g\rangle _{\omega}=\int_{\D} f(\zeta)\,\overline{g(\zeta)}\,\omega(\zeta)\,d\mu(\zeta), \qquad f,g\in A^2_{\omega}.
\]
\end{lemma}

\begin{proof}
Being $\omega$ a radial weight, the polynomials are dense in $A^2_{\omega}$ and we may assume that $g$ is an holomorphic polynomial. Because of Theorem \ref{cdecarl1}, the condition implies that $d\nu=\omega \,d\mu$ is a Carleson measure for $A^2_{\omega}$. Then
\[
\begin{split}
\int_{\D} &\left ( \int_{\D} |f(\zeta)|\,|K_ z(\zeta)|\,\omega(\zeta) \,d\mu(\zeta)\right ) |g(z)|\,\omega(z) \,dA(z)\\
& \le \|f\|_{L^2(\nu)} \int_{\D} \|K_ z\|_{L^2(\nu)} \, |g(z)|\,\omega(z) \,dA(z)
\\
& \lesssim \|f\|_{A^2_{\omega}} \cdot \|g\|_{\infty} \int_{\D} \|K_ z\|_{A^2_{\omega}}\,\omega(z) \,dA(z)
\\
& \lesssim \|f\|_{A^2_{\omega}} \cdot \|g\|_{\infty} \int_{\D} \frac{\omega(z)^{1/2}}{\tau(z)} \,dA(z),
\end{split}
\]
and this is finite (see \cite[Lemma 2.3]{PP1} for example). Thus, Fubini's theorem gives
\[
\begin{split}
\langle T_{\mu} f, g\rangle _{\omega}&=\int_{\D} \left ( \int_{\D} f(\zeta)\,\overline{K_ z(\zeta)}\,\omega(\zeta) \,d\mu(\zeta)\right ) \overline{g(z)}\,\omega(z) \,dA(z)\\
\\
&= \int_{\D} f(\zeta) \left (\int_{\D} \overline{g(z)}\, K_{\zeta}(z) \,\omega(z) \,dA(z) \right) \,\omega(\zeta)\,d\mu(\zeta)\\
\\
&=\int_{\D} f(\zeta) \,\overline{ \langle g, K_{\zeta} \rangle _{\omega}} \,\omega(\zeta)\,d\mu(\zeta)=\int_{\D} f(\zeta) \,\overline{g(\zeta)} \,\omega(\zeta)\,d\mu(\zeta).
\end{split}
\]
\end{proof}

\section{Proof of Theorem \ref{TOPQ}}\label{sec 4:2}

Recall that, for $\delta\in(0,m_\tau),$  the averaging function $\widehat{\mu}_{\delta}$ is defined  on $\D$  by
\[
\widehat{\mu}_{\delta}(z) : = \frac{\mu(D(\delta\tau(z))}{\tau(z)^2},\qquad z\in \D.
\]
\subsection{Boundedness}
Assume first  that $ T_\mu$   is bounded on $A^2_{\omega}$. For  fixed $a \in \D,$  one has
\[
T_\mu K_a(a) = \int_{\D}{|K_a(z)|^2 \, \omega(z) \ d\mu(z)}.
\]
By \eqref{RK-Diag}, there is $\delta\in (0,m_{\tau})$ such that $|K_a(z)|\asymp\|K_z\|_{A^2_{\omega}}\cdot \|K_a\|_{A^2_{\omega}},$  for every $z \in D(\delta\tau(a))$. This together with  the norm estimate given in \eqref{Eq-NE} and the fact that $\tau(z)\asymp \tau(a)$ for $z\in D(\delta \tau(a))$ gives
\begin{equation}\label{Eq-Tm1}
\begin{split}
T_\mu K_a(a) &\geq \int_{D(\delta\tau(a))}|K_a(z)|^2\, \omega(z) \,d\mu(z) \\
\\
&\gtrsim \int_{D(\delta\tau(a))} \|K_z\|^2_{A^2_{\omega}}\,\|K_a\|^2_{A^2_{\omega}} \,\omega(z) \,d\mu(z)\\
\\
&\asymp \frac{\mu(D(\delta\tau(a))}{\omega(a)\,\tau(a)^4} =\frac{\widehat{\mu}_{\delta}(a)}{\omega(a)\,\tau(a)^2}.
\end{split}
\end{equation}
Therefore, by Lemma \ref{subHarmP} and the estimate of the norm of the reproducing kernels, we obtain

\begin{equation}\label{TOPQ-1}
\begin{split}
\widehat{\mu}_{\delta}(a)
&\lesssim \omega(a)\, \tau(a)^{2}\, |T_\mu K_a(a)| \leq \omega(a)^{1/2}\,\tau(a) \,\,\| T_\mu K_a\|_{A^2_{\omega}}\\
\\
&\leq\omega(a)^{1/2}\,\tau(a)\,\, \|T_\mu\| \cdot \|K_a\|_{A^2_{\omega}}
\lesssim \|T_\mu\|.
\end{split}
\end{equation}
\mbox{}
\\
Conversely, suppose that \eqref{tauomega1} holds. Let $f,g\in A^2_{\omega}$. By Lemma \ref{LF1} and because $d\nu=\omega d\mu$ is a Carleson measure for $A^2_{\omega}$ with $\|I_{\nu}\|^2\asymp C_{\mu}:=\sup_{z\in \D} \widehat{\mu}_{\delta}(z)$  (see Theorem \ref{cdecarl1}), we have
\[
\big | \langle T_{\mu} f, g\rangle _{\omega} \big | \le \int_{\D} |f(z)|\,|g(z)|\,d\nu(z) \le \|f\|_{L^2(\nu)}\cdot \|g\|_{L^2(\nu)}\lesssim C_{\mu} \|f\|_{A^2_{\omega}}\cdot \|g\|_{A^2_{\omega}}.
\]
 This shows that $T_{\mu}$ is bounded on $A^2_{\omega}$ with $\|T_{\mu}\|\lesssim C_{\mu}$ finishing the proof.

\subsection{Compactness}

Let $k_{z}$ be the normalized reproducing kernels in $A^2_{\omega}$. From \eqref{TOPQ-1} in the proof of the boundedness part, and the estimate for  $\|K_ z\|_{A^2_{\omega}}$, we have
\[
\widehat{\mu}_{\delta}(z)\lesssim \|T_{\mu} k_{z}\|_{A^2_{\omega}}.
\]
 From Lemma \ref{subHarmP}, it is easy to see that $k_{z}$ converges to zero weakly  as $|z|\rightarrow 1^{-}$. Thus, if $T_{\mu}$ is compact, from Theorem 1.14 in \cite{Zhu}, we obtain \eqref{CompCond}.\\

Conversely, suppose \eqref{CompCond} holds, and let $\lbrace f_n\rbrace$  be a  sequence in $A^2_{\omega}$ converging to zero weakly. To prove compactness, we must show that $\|T_\mu f_n\|_{ A^2 _{\omega}}\rightarrow 0$. By the proof of the boundedness, we have
\[
\|T_\mu f_n\|_{ A^2 _{\omega}}\lesssim  \|f_n\|_{L^2( \nu)},
\]
with $d\nu := \omega d\mu. $ By Theorem \ref{cdecarl1}, our assumption \eqref{CompCond} implies that  $I_{\nu}: A^2_{\omega}\longrightarrow  L^2(\D, d\nu) $ is compact, which implies that $\|f_n\|_{L^2( \nu)}$ tends to zero.
Hence $\|T_\mu f_ n \|_{A^2_{\omega}}\rightarrow 0$ proving that $T_{\mu}$ is compact. The proof is complete.

\section[Schatten classes]{Membership in Schatten classes}\label{sec3:4}

In this section, we are going to describe those positive Borel measures $\mu$ for which the Toeplitz operator $T_{\mu}$ belongs to the the Schatten ideal $S_ p(A^2_{\omega})$, for weights $\omega \in \mathcal{W}$. In order to obtain such a characterization, we need to introduce first some concepts. \\

We  define the  $\omega$-Berezin transform $B_{\omega}\mu$ of the measure $\mu$  as
\[
B_{\omega}\mu(z) : = \int_{\D}{ |k_z(\xi)|^2\,\omega(\xi)\, d\mu(\xi)},\qquad z\in \D,
\]
where $k_ z$ are the normalized reproducing kernels in $A^2_{\omega}$. We also recall that, for $\delta\in (0,m_\tau),$  the averaging function $\widehat{\mu}_{\delta}$ is given  by
\[
 \widehat{\mu}_{\delta}(z): = \frac{\mu(D(\delta \tau(z))}{ \tau(z)^2}, \qquad z\in \D.
\]
We also consider the measure $\lambda_{\tau}$ given by
\[
d\lambda_{\tau}(z)=\frac{dA(z)}{\tau(z)^2},\qquad z\in \D.
\]
\begin{proposition}\label{SchatC}
Let $ 1 \le p<\infty,$ and $\omega \in \mathcal{W}$.   The following conditions are equivalent :
\begin{enumerate}
\item[$(a)$] The function $B_{\omega}\mu$ is in $L^p(\D,d\lambda_{\tau}).$
\item[$(b)$] The function $\widehat{\mu}_{\delta}$ is in $L^p(\D,d\lambda_{\tau})$ for any $\delta\in (0,m_{\tau})$ small enough.
\item[$(c)$] The sequence $\lbrace\widehat{\mu}_{\delta}(z_n)\rbrace$ is in $\ell^p$ for any $(\delta,\tau)$-lattice $\{z_ n\}$ with $\delta\in (0,\frac{m_{\tau}}{4})$ sufficiently small.
\end{enumerate}
\end{proposition}

\begin{proof}

$(a)\Rightarrow (b)$. By Theorem \ref{RK-PE}, for all $\delta \in (0,m_{\tau})$ sufficiently small, one has
\begin{equation*}
|K_ z(\zeta)| \asymp \|K_ z\|_{A^2(\omega)}\cdot \|K_{\zeta}\|_{A^2(\omega)} ,\qquad \zeta \in D(\delta \tau(z))
\end{equation*}
Then
\begin{displaymath}\label{MuC}
\begin{split}
B_{\omega}\mu(z)&= \int_{\D}{|k_z(\zeta)|^2 \,\omega(\zeta)\,d\mu(\zeta)}\ge \|K_ z\|_{A^2_{\omega}}^{-2}\int_{D(\delta \tau(z))}{|K_z(\zeta)|^2 \,\omega(\zeta)\,d\mu(\zeta)}\\
\\
& \asymp \int_{D(\delta \tau(z))}{\|K_ {\zeta}\|_{A^2_{\omega}}^2 \,\omega(\zeta)\,d\mu(\zeta)}\asymp  \widehat{\mu}_{\delta}(z) .
\end{split}
\end{displaymath}
Since $B_{\omega}\mu$ is in $L^p(\D,d\lambda_{\tau}),$ this gives (b). \\
\\
$(b)\Rightarrow(c)$. Since
$\widehat{\mu}_{\delta}(z_n) \lesssim \widehat{\mu}_{4\delta}(z),$ for $z\in D(\delta\tau(z_n)),$ then
\begin{displaymath}
\begin{split}
\sum_{n} \widehat{\mu}_{\delta}(z_n)^p &\lesssim \sum_{n} \int_{D(\delta\tau(z_n))} \!\!\widehat{\mu}_{4\delta}(z)^p\, \frac{dA(z)}{\tau(z)^2}
\lesssim  \int_{\D} \widehat{\mu}_{4\delta}(z)^p\,d\lambda_{\tau}(z).
\end{split}
\end{displaymath}
\mbox{}
\\
$(c)\Rightarrow (a)$.  We have
\begin{displaymath}
B_{\omega}\mu(z)\leq \|K_ z\|^{-2}_{A^2_{\omega}}\sum_{n} \int_{D(\delta\tau(z_n))} \! \! |K_z(s)|^2\,\omega(s)\, d\mu(s).
\end{displaymath}
\mbox{}
\\
We use Lemma \ref{subHarmP} in order to obtain

\[
\begin{split}
\int_{D(\delta\tau(z_n))} \! \! |K_z(s)|^2\,\omega(s)\, d\mu(s)& \lesssim \int_{D(\delta\tau(z_n))} \! \!
\bigg (\frac{1}{\tau(s)^2} \int_{D(\delta \tau(s))} \!\!|K_z(\xi)|^2\,\omega(\xi)\,dA(\xi) \bigg ) d\mu(s)\\
\\
& \lesssim \bigg ( \int_{D(3\delta \tau(z_ n))} \!\!|K_z(\xi)|^2\,\omega(\xi)\,dA(\xi) \bigg ) \,\,\widehat{\mu}_{\delta}(z_ n).
\end{split}
\]
\mbox{}
\\
If $p>1$, by H\"{o}lder's inequality,

\[
\begin{split}
\Bigg (\sum_{n} \int_{D(\delta\tau(z_n))} \! \! &|K_z(s)|^{2}\,\omega(s)\, d\mu(s)\Bigg )^p\\
\\
&\lesssim \|K_ z\|^{2(p-1)}_{A^2_{\omega}} \sum_ n \bigg ( \int_{D(3\delta \tau(z_ n))} \!\!|K_z(\xi)|^2\,\omega(\xi)\,dA(\xi) \bigg ) \,\,\widehat{\mu}_{\delta}(z_ n)^p.
\end{split}
\]
\mbox{}
\\
This gives

\[
\int_{\D} B_{\omega} \mu(z)^p \,d\lambda_{\tau}(z)\lesssim \sum_ n \widehat{\mu}_{\delta}(z_ n)^p \int_{D(3\delta \tau(z_ n))} \!\!\bigg(\int_{\D} |K_{\xi}(z)|^2\,\|K_ z\|^{-2}_{A^2_{\omega}} d\lambda_{\tau}(z)\bigg)\, \omega(\xi)\,dA(\xi).
\]
\mbox{}
\\
Since $\|K_ z\|^{2}_{A^2_{\omega}}\asymp \tau(z)^{-2}\omega(z)^{-1}$, we have

\[
\int_{\D} |K_{\xi}(z)|^2\,\|K_ z\|^{-2}_{A^2_{\omega}} d\lambda_{\tau}(z)\asymp \|K_{\xi}\|^2_{A^2(\omega)}\asymp \tau(\xi)^{-2}\omega(\xi)^{-1}.
\]
\mbox{}
\\
Putting this estimate in the previous inequality, we finally get

\[
\int_{\D} B_{\omega} \mu(z)^p \,d\lambda_{\tau}(z)\lesssim \sum_ n \widehat{\mu}_{\delta}(z_ n)^p, \qquad p\ge 1.
\]
This finishes the proof.
\end{proof}
\mbox{}
\\
\begin{rk}
It should be observed that the equivalence of (b) and (c) in Proposition \ref{SchatC} continues to hold for $0<p<1$ as well as the implication (a) implies (b). In order to get equivalence with condition (a) even in the case $0<p<1$, it seems that one needs $L^p$-integral estimates for reproducing kernels, estimates that are not available nowadays.
\end{rk}
\mbox{}
\\

Next Lemma is the analogue to our setting of a well known result for standard Bergman spaces.

\begin{lemma}\label{KL-T}
Let $\omega \in \mathcal{W}$, and $T$ be a positive operator on $A^ 2_{\omega}$. Let $\widetilde{T}$ be the Berezin transform of the operator $T$ defined by
\[\widetilde{T}(z)=\langle Tk_ z,k_ z\rangle_{\omega},\quad z\in \D.\]
\begin{enumerate}
\item[(a)]  Let $0<p\le 1$. If $\widetilde{T} \in L^p(\D,d\lambda_{\tau})$ then $T$ is in $S_ p$.

\item[(b)] Let $p\ge 1$. If $T$ is in $S_ p$ then $\widetilde{T} \in L^p(\D,d\lambda_{\tau})$.
\end{enumerate}
\end{lemma}

\begin{proof}
Let $p>0$. The positive operator $T$ is in $S_ p$ if and only if $T^p$ is in the trace class $S_ 1$. Fix an orthonormal basis $\{e_ k\}$ of $A^2_{\omega}$. Since $T^p$ is positive, it belongs to the trace class if and only if
$\sum _ k \langle T^p e_ k,e_ k \rangle _{\omega} <\infty.$
Let $S=\sqrt{T^p}$. Then
\[\sum _ k \langle T^p e_ k,e_ k \rangle _{\omega} =\sum_ k \|S e_ k\|_{A^2_{\omega}}^2.\]
Now, by  Fubini's theorem and Parseval's identity, we have
\begin{displaymath}
\begin{split}
\sum_ k \|S e_ k\|_{A^2_{\omega}}^2&=\sum_ k \int_{\D}|S e_ k(z) |^2\,\omega(z)\,dA(z)
=\sum_ k \int_{\D}|\langle S e_ k, K_ z \rangle_{\omega} |^2\,\omega(z)\,dA(z)\\
\\
&=\int_{\D}\left (\sum_k |\langle  e_ k, SK_ z \rangle_{\omega} |^2\right )\,\omega(z)\,dA(z)
=\int_{\D}\|SK_ z\|_{A^2_{\omega}}^2\,\omega(z)\,dA(z)\\
\\
&=\int_{\D}\langle T^p K_ z,K_ z\rangle_{\omega}\,\omega(z)\,dA(z)
=\int_{\D}\langle T^p k_ z,k_ z\rangle_{\omega}\|K_ z\|_{A^2_{\omega}}^{2}\,\omega(z)\,dA(z)\\
\\
&\asymp \int_{\D}\langle T^p k_ z,k_ z\rangle_{\omega}\,d\lambda_{\tau}(z).
\end{split}
\end{displaymath}
\mbox{}
\\
Hence, both (a) and (b) are consequences of the inequalities (see \cite[Proposition 1.31]{Zhu})

\[ \langle T^p k_ z,k_ z\rangle_{\omega} \le  \left [\langle T k_ z,k_ z\rangle_{\omega} \right ]^p=[\widetilde{T}(z)]^p, \qquad 0<p\le 1\]
and

\[[\widetilde{T}(z)]^p=\left [\langle T k_ z,k_ z\rangle_{\omega} \right ]^p \le \langle T^p k_ z,k_ z\rangle_{\omega},\qquad p\ge 1.\]
\mbox{}
\\
This finishes the proof of the lemma.
\end{proof}
\mbox{}
\\

\begin{proposition}\label{TopMu}
Let $\omega \in \mathcal{W}$. If $0 < p \leq 1$ and $B_{\omega}\mu$ is in $L^p(\D,d\lambda_{\tau}),$ then  $T_\mu$ belongs to $S_p(A^2_{\omega}).$ Conversely, if $p\ge 1$ and $T_\mu$ is in $S_p(A^2_{\omega})$, then $B_{\omega}\mu \in L^p(\D,d\lambda_{\tau}).$
\end{proposition}

\begin{proof}
If $B_{\omega}\mu$ is in $L^p(\D,d\lambda_{\tau}),$ then it is easy to see that $T_{\mu}$ is bounded on $A^2_{\omega}$ (just use the discrete version in Proposition \ref{SchatC} to see that the condition in Theorem \ref{TOPQ} holds). Therefore,  the result is a consequence of Lemma \ref{KL-T} since  $\widetilde{T_\mu}(z)= B_{\omega}\mu(z)$.
\end{proof}

\mbox{}
\\
Now we are almost ready for the characterization of Schatten class Toeplitz operators, but we need first some technical lemmas on properties of lattices. We use the notation
\[
d_{\tau}(z,\zeta)=\frac{|z-\zeta|}{\min(\tau(z),\tau(\zeta))},\qquad z,\zeta \in \D.
\]

\begin{lemma}\label{L-M}
Let $\tau \in \mathcal{L}$, and $\{z_ j\}$ be a $(\delta,\tau)$-lattice on $\D$. For each $\zeta \in \D$, the set
\[ D_ m(\zeta)=\big \{z\in \D: d_{\tau}(z,\zeta)<2^m \delta \big \}
\]
contains at most $K$ points of the lattice, where $K$ depends on the positive integer $m$ but not on the point $\zeta$.
\end{lemma}

\begin{proof}
Let $K$ be the number of points of the lattice contained in $D_ m(\zeta)$. Due to the Lipschitz condition (B), we have

\[
\tau(\zeta) \le \tau(z_ j)+c_ 2 |\zeta-z_ j| \le (1+c_ 2 2^m \delta)\, \tau(z_ j)=C_ m \,\tau(z_ j).
\]
Then

\[
 K\cdot \tau(\zeta)^2 \le C_ m^2 \sum_{z_ j\in D_ m(\zeta)} \!\!\! \tau(z_ j)^2 \lesssim C_ m^2 \cdot Area \left (\bigcup_{z_ j\in D_ m(\zeta)} \!\! D\Big (\frac{\delta}{4} \tau(z_ j)\Big )  \right ).
\]
\mbox{}
\\
As done before, we also have $\tau(z_ j)\le C_ m \tau(\zeta)$, if $z_ j\in D_ m(\zeta)$. From this we easily see that
\[
D\Big (\frac{\delta}{4}\, \tau(z_ j)\Big )\subset D \Big ( c 2^m\delta  \tau(\zeta)\Big )
\]
for some constant $c$. Since the sets $\{D(\frac{\delta}{4}\,\tau(z_ j))\}$ are pairwise disjoints, we have

\[
\bigcup_{z_ j\in D_ m(\zeta)} \!\! D\Big (\frac{\delta}{4} \tau(z_ j)\Big ) \subset D \Big ( c 2^m\delta  \tau(\zeta)\Big ).
\]
Therefore, we get

\[K \cdot \tau(\zeta)^2 \le C_ m^2 \cdot Area \left (D \Big ( c 2^m\delta  \tau(\zeta)\Big ) \right )\lesssim C_ m^2 2^{2m} \tau(\zeta)^2,
\]
\mbox{}
\\
that implies $K\le C 2^{4m}$.
\end{proof}
\mbox{}
\\

Next, we use the result just proved to decompose any $(\delta,\tau)$-lattice into a finite number of ``big" separated subsequences.

\begin{lemma}\label{L-part}
Let $\tau\in \mathcal{L}$ and $\delta\in (0,m_{\tau})$. Let $m$ be a positive integer. Any $(\delta,\tau)$-lattice $\{z_ j\}$ on $\D$ can be partitioned into $M$ subsequences such that, if $a_ j$ and $a_ k$ are different points in the same subsequence, then $d_{\tau}(a_ j,a_ k)\ge 2^m \delta$.
\end{lemma}

\begin{proof}
Let $K$ be the number given by Lemma \ref{L-M}. From the lattice $\{z_ j\}$ extract a maximal $(2^m \delta)$-subsequence, that is, we select one point $\xi_ 1$ in our lattice, and then we continue selecting points $\xi_ n$ of the lattice so that $d_{\tau}(\xi_ n, \xi)\ge 2^m \delta$ for all previous selected point $\xi$. We stop once the subsequence is maximal, that is, when all the remaining points $x$ of the lattice satisfy $d_{\tau}(x,\xi_ x)<2^m \delta$ for some $\xi_ x$ in the subsequence. With the remaining points of the lattice we extract another maximal $(2^m \delta)$-subsequence, and we repeat the process until we get $M=K+1$ maximal $(2^m \delta)$-subsequences. If no point of the lattice is left, we are done. On the other hand, if a point $\zeta$ in the lattice is left, this means that there are $M=K+1$ distinct points $x_{\zeta}$ (at least one for each subsequence) in the lattice with $ d_{\tau}(\zeta,x_{\zeta})<2^m \delta$, in contradiction with the choice of $K$ from Lemma \ref{L-M}. The proof is complete.
\end{proof}
\mbox{}
\\

Now we are ready for the main result of this Section, that characterizes the membership in the Schatten ideals of the Toeplitz operator acting on $A^2_{\omega}$.

\begin{theorem}\label{schtCla}
Let  $\omega \in \mathcal{W}$ and $ 0 <p<\infty.$  The following conditions are equivalent:
\begin{enumerate}
\item[$(a)$] The Toeplitz operator $T_\mu$ is in  $S_p(A^2_{\omega}).$
\item[$(b)$] The function $\widehat{\mu}_\delta$ is in $L^p(\D,d\lambda_{\tau})$ for $\delta \in (0, m_\tau) $ sufficiently small.
\item[$(c)$] The sequence $\lbrace\widehat{\mu}_{\delta}(z_n)\rbrace$ is in $\ell^p$ for any $\delta>0$ small enough.
\end{enumerate}
Moreover, when $p\ge 1$, the previous conditions are also equivalent to:
\begin{enumerate}
\item[$(d)$] The function $B_{\omega}\mu$ is in $L^p(\D,d\lambda_{\tau}).$
\end{enumerate}
\end{theorem}

\begin{proof}
By Proposition \ref{SchatC} and the remark following it, the statements $(b)$ and $(c)$ are equivalent, and when $p\ge 1$, they are also equivalent with condition $(d)$. Hence, according to Proposition \ref{TopMu}, the result is proved for $p=1$ and we have the implication $(a)$ implies $(b)$ for $p>1$. Moreover, the implication $(c)$ implies $(a)$ for $0<p<1$ is proved in \cite{LR2} (the conditions on the weights are slightly different, but the same proof works for our class). Thus, it remains to prove that $(b)$ implies $(a)$ for $p> 1$, and that $(a)$ implies $(c)$ when $0<p<1$. \\

Let $1<p<\infty$,  and assume that $\widehat{\mu}_{\delta} \in L^p(\D,d\lambda_{\tau})$ with $\delta\in (0,m_{\tau})$ small enough. It is not difficult to see, using the equivalent discrete condition in $(c)$ together with Theorem \ref{TOPQ}, that $T_{\mu}$ must be compact. For any orthonormal set $\{e_ n\}$ of $A^2_{\omega}$, we have
\begin{equation}\label{Sp1}
\sum_ n \langle T_{\mu} e_ n,e_ n \rangle _{\omega}^p =\sum_ n \left ( \int_{\D} |e_ n(z)|^2 \,\omega(z)\,d\mu(z) \right )^p.
\end{equation}
By Lemma \ref{subHarmP} and Fubini's theorem,
\[
\begin{split}
\int_{\D} |e_ n(z)|^2 \,\omega(z)\,d\mu(z) &\lesssim \int_{\D} \bigg(\frac{1}{\tau(z)^2}\int_{D(\delta \tau(z))} |e_ n(\zeta)|^2 \,\omega(\zeta)\,dA(\zeta)\bigg )\,d\mu(z)\\
\\
&\lesssim \int_{\D} |e_ n(\zeta)|^2 \,\omega(\zeta)\,\widehat{\mu}_{\delta}(\zeta)\,dA(\zeta).
\end{split}
\]
Since $p>1$ and $\|e_ n\|_{A^2_{\omega}}=1$, we can apply H\"{o}lder's inequality to get

\[
\left ( \int_{\D} |e_ n(z)|^2 \,\omega(z)\,d\mu(z) \right )^p \lesssim \int_{\D} |e_ n(\zeta)|^2 \,\omega(\zeta)\,\widehat{\mu}_{\delta}(\zeta)^p\,dA(\zeta).
\]
\mbox{}
\\
Putting this into \eqref{Sp1} and taking into account that $\|K_{\zeta}\|^2_{A^2_{\omega}}\,\omega(\zeta)\asymp \tau(\zeta)^{-2}$, we see that

\[
\begin{split}
\sum_ n \langle T_{\mu} e_ n,e_ n \rangle _{\omega}^p & \lesssim \int_{\D} \bigg (\sum_ n |e_ n(\zeta)|^2 \bigg ) \,\omega(\zeta)\,\widehat{\mu}_{\delta}(\zeta)^p\,dA(\zeta)\\
\\
& \le \int_{\D} \|K_{\zeta}\|^2_{A^2_{\omega}} \,\omega(\zeta)\,\widehat{\mu}_{\delta}(\zeta)^p\,dA(\zeta)\\
\\
& \asymp \int_{\D} \widehat{\mu}_{\delta}(\zeta)^p\,d\lambda_{\tau} (\zeta).
\end{split}
\]
By \cite[Theorem 1.27]{Zhu} this proves that $T_{\mu}$ is in $S_ p$ with
$\|T_{\mu}\|_{S_ p} \lesssim \|\widehat{\mu}_{\delta}\|_{L^p(\D,d\lambda_{\tau})}.$
\mbox{}
\\

Next, let $0<p<1$, and suppose that $T_\mu \in S_p(A^2_{\omega}).$ We will prove that $(c)$ holds. The method for this proof has his roots in previous work of S. Semmes \cite{Sem} and D. Luecking \cite{Lue}. Let $\lbrace z_n\rbrace$ be a $(\delta,\tau)$-lattice on $\D$. We want to show that $\{\widehat{\mu}_{\delta}(z_ n)\}$ is in $\ell^p$. To this end, we fix a large positive integer $m\ge 2$ and apply Lemma \ref{L-part} to partition the lattice $\{z_ n\}$ into $M$ subsequences such that any two distinct points $a_ j$ and $a_ k$ in the same subsequence satisfy $d_{\tau}(a_ j,a_ k)\ge 2^m \delta$. Let $\{a_ n \}$ be such a subsequence and consider the measure

\[
\nu=\sum_ n \mu \chi_ n,
\]
\mbox{}
\\
where $\chi_ n$ denotes the characteristic function of $D(\delta \tau(a_ n))$. Since $m\ge 2$, the disks $D(\delta \tau(a_ n))$ are pairwise disjoints.
Since $T_{\mu}$ is in $S_ p$ and $0\le \nu \le \mu$, then $0\le T_{\nu}\le T_{\mu}$ which implies that $T_{\nu}$ is also in  $S_ p$. Moreover, $\|T_{\nu}\|_{S_ p}\le \|T_{\mu}\|_{S_ p}$.
Fix an orthonormal basis $\lbrace e_n\rbrace$ for $A^2_{\omega} $ and define an operator $B$ on  $A^2_{\omega} $  by
\[
B \left( \sum_ n \lambda_ n\,  e_n \right ) = \sum_ n \lambda_ n \,f_{a_n},
\]
where $f_{a_n}=F_{a_ n,N}/\tau(a_ n)$ and $F_{a_ n,N}$ are the functions appearing in Lemma \ref{Borichevlemma} with $N$ taken big enough so that $3Np-4>2p$. By  \cite[Proposition 2]{PP1}, the operator $B$ is bounded. Since $T_{\nu}\in S_ p$, the operator $T=A^*T_{\nu} A$ is also in $S_ p$, with
\[
\|T\|_{S_ p}\le \|B\|^2 \cdot \|T_{\nu}\|_{S_ p}.
\]
We split the operator $T$ as $T = D + E,$ where $D$ is the diagonal operator on  $A^2_{\omega} $ defined by
\[
 Df = \sum_{n=1}^{\infty} \langle Te_n,e_n\rangle_{\omega}\, \langle f,e_n\rangle _{\omega}\, e_n,   \qquad f\in A^2_{\omega},
\]
and $E = T - D.$ By the triangle inequality,
\begin{equation}\label{DTDE}
\|T\|^p_{S_p} \geq \|D\|^p_{S_p} - \|E\|^p_{S_p}.
\end{equation}
\mbox{}
\\
Since $D $ is positive diagonal operator,   we have
 \begin{equation*}
\begin{split}
 \|D\|^p_{S_p}& = \sum_{n} \langle Te_n, e_n\rangle _{\omega} ^p=  \sum_{n} \langle T_\nu f_{a_n}, f_{a_n}\rangle _{\omega} ^p\\
\\
&=  \sum_{n} \left(\int_{\D} |f_{a_ n}(z)|^2\,\omega(z)\,d\nu(z) \right )^p \\
\\
&\ge \sum_{n} \left(\int_{D(\delta\tau(a_ n))} \frac{|F_{a_ n,N}(z)|^2}{\tau(a_ n)^2}\,\omega(z)\,d\mu(z) \right )^p.
\end{split}
\end{equation*}
Hence, by Lemma \ref{Borichevlemma}, there is a positive constant $C_ 1$ such that
\begin{equation}\label{NDD}
\|D\|^p_{S_p} \ge C_ 1 \sum_ n \widehat{\mu}_{\delta}(a_ n)^p.
\end{equation}
On the other hand, since $0<p<1$, by \cite[Proposition 1.29]{Zhu} and Lemma \ref{LF1}  we have
\begin{equation}\label{DiagDec}
\begin{split}
 \|E\|^p_{S_p}& \leq  \sum_{n} \sum_{k}  \langle Ee_n, e_k\rangle _{\omega} ^p = \sum_{n,k:k\neq n}  \langle T_\nu f_{a_n}, f_{a_k}\rangle_{\omega}^p\\
 \\
&\leq \sum_{n,k:k\neq n}  \bigg(\int_{\D} | f_{a_n}(\xi)|\,| f_{a_k}(\xi)|\,\omega(\xi)\, d\nu(\xi)\bigg)^p\\
\\
&=  \sum_{n.k: k\neq n}  \bigg(\sum_{j} \int_{D(\delta\tau(a_j))} | f_{a_n}(\xi)|\,| f_{a_k}(\xi)| \,\omega(\xi)\,d\mu(\xi)\bigg)^p.
\end{split}
\end{equation}

If $n\neq k,$ then $d_{\tau}(a_ n,a_ k) \geq 2^m\delta.$ Thus,  for $\xi\in D(\delta \tau(a_ j))$, is not difficult to see that either
\[
d_{\tau}(\xi,a_ n)\ge 2^{m-2}\delta \quad \textrm{or}\quad  d_{\tau}(\xi,a_ k)\ge 2^{m-2}\delta.
 \]
 Indeed, since $n\ne k$, then either $d_{\tau}(a_ n,a_ j) \geq 2^m\delta$ or $d_{\tau}(a_ j,a_ k) \geq 2^m\delta.$ Suppose that $d_{\tau}(a_ n,a_ j) \geq 2^m\delta$. If $d_{\tau}(\xi,a_ n)< 2^{m-2}\delta$, then
\[
\begin{split}
|a_ n-a_ j| &  \le |a_ n-\xi|+|\xi-a_ j| <2^{m-2}\delta \min(\tau(a_ n),\tau(\xi))+\delta \tau(a_ j)
\\
&\le 2^{m-1}\delta \min(\tau(a_ n),\tau(a_ j))+\delta \tau(a_ j).
\end{split}
\]
This directly gives a contradiction if $\min(\tau(a_ n),\tau(a_ j))=\tau(a_ j)$. In case that $\min(\tau(a_ n),\tau(a_ j))=\tau(a_ n),$ using the Lipschitz condition (B)  we get
\[
|a_ n-a_ j|< 2^{m-1}\delta \min(\tau(a_ n),\tau(a_ j))+\delta \tau(a_ n)+c_ 2 \delta |a_ n-a_ j|
\]
Since $c_ 2 \delta \le 1/4$, and $m\ge 2$, we see that this implies
\[
d_{\tau}(a_ n,a_ j)<\frac{4}{3}(2^{m-1}+1) \delta\le 2^m \delta.
\]
Thus, without loss of generality, we assume that  $d_{\tau}(\xi,a_ n)\ge 2^{m-2}\delta$. For any $n$ and $k$ we write

\[
I_{n k}(\mu)  =   \sum_{j}\int_{D(\delta\tau(a_j))} | f_{a_n}(\xi)|\,| f_{a_k}(\xi)|\,\omega(\xi)\, d\mu(\xi).
\]
\mbox{}
\\
With this notation and taking into account \eqref{DiagDec}, we have

\begin{equation}\label{Eq-SC-0}
\|E\|^p_{S_p}\le \sum_{n.k: k\neq n} I_{nk}(\mu)^p.
\end{equation}
By Lemma \ref{Borichevlemma}, we have

\[
|F_{a_ n,N}(\xi)| \lesssim \,d_{\tau}(\xi,a_ n)^{-3N}.
\]
\mbox{}
\\
Apply this inequality raised to the power $1/2$, together with the fact that $d_{\tau}(\xi,a_ n)\ge 2^{m-2}\delta$ to get

\begin{equation}\label{Eq-SC-1}
\begin{split}
| f_{a_n}(\xi)| & = \frac{ |F_{a_ n,N}(\xi)|^{1/2}}{\tau(a_ n)} \,\, |F_{a_ n,N}(\xi)|^{1/2} \lesssim 2^{-\frac{3Nm}{2}} \, \frac{|f_{a_ n}(\xi)|^{1/2} }{\tau(a_ n)^{1/2}}.
\end{split}
\end{equation}
\mbox{}
\\
We also have
\begin{equation}\label{Eq-SC-2}
| f_{a_k}(\xi)|  \le  |f_{a_ k}(\xi)|^{1/2} \cdot \|K_{\xi}\|^{1/2}_{A^2_{\omega}}
\end{equation}
\mbox{}
\\
Putting \eqref{Eq-SC-1} and \eqref{Eq-SC-2} into the definition of $I_{nk}(\mu)$, and using the norm estimate
\[
\|K_{\xi}\|_{A^2_{\omega}}\asymp \tau(\xi)^{-1}\,\omega(\xi)^{-1/2},
\]
we obtain

\[
I_{nk} (\mu)\lesssim \frac{2^{-\frac{3Nm}{2}}}{\tau(a_ n)^{1/2}}\,\sum_{j}\frac{1}{\tau(a_ j)^{1/2}}\int_{D(\delta\tau(a_j))} | f_{a_n}(\xi)|^{1/2}\,| f_{a_k}(\xi)|^{1/2}\,\omega(\xi)^{1/2}\, d\mu(\xi).
\]

By Lemma \ref{subHarmP}, for $\xi\in D(\delta \tau(a_ j))$, one has

\[
\begin{split}
| f_{a_n}(\xi)|^{1/2}\,\omega(\xi)^{1/4} & \lesssim \left (\frac{1}{\tau(\xi)^2} \int_{D(\delta \tau(\xi))} | f_{a_n}(z)|^{p/2}\,\omega(z)^{p/4}\,dA(z)\right )^{1/p}
\\
\\
& \lesssim \tau(a_ j)^{-2/p} \,S_ n(a_ j)^{1/p},
\end{split}
\]
with
\[
S_{n}(x)  = \int_{D(3\delta\tau(x))} |f_{a_ n}(z)|^{p/2}\omega(z)^{p/4} dA(z).
\]
In the same manner we also have

\[
|f_{a_k}(\xi)|^{1/2}\,\omega(\xi)^{1/4}\lesssim \tau(a_ j)^{-2/p} \,S_ k(a_ j)^{1/p}.
\]
\mbox{}
\\
Therefore, there is a positive constant $C_ 2$ such that

\[
\begin{split}
I_{nk} (\mu)& \le C_ 2\cdot  \frac{2^{-\frac{3Nm}{2}}}{\tau(a_ n)^{1/2}}\,\sum_{j}\frac{\tau(a_ j)^{-4/p}}{\tau(a_ j)^{1/2}}\,\, S_ n (a_ j)^{1/p}\cdot S_ k(a_ j)^{1/p} \,\,\mu \Big(D(\delta\tau(a_j))\Big) \\
\\
&=C_ 2\cdot  \frac{2^{-\frac{3Nm}{2}}}{\tau(a_ n)^{1/2}}\,\sum_{j}\tau(a_ j)^{3/2-4/p}\cdot  S_ n (a_ j)^{1/p}\cdot S_ k(a_ j)^{1/p} \cdot \widehat{\mu}_{\delta}(a_ j).
\end{split}
\]
\mbox{}
\\
Since $0<p<1$ and $2^{-3Nm/2}\le 2^{-m}$, we get

\[
I_{nk} (\mu) ^p\le C_ 2^p\cdot  \frac{2^{-mp}}{\tau(a_ n)^{p/2}}\,\sum_{j}\tau(a_ j)^{\frac{3p}{2}-4}\cdot S_ n (a_ j)\cdot S_ k(a_ j) \cdot \widehat{\mu}_{\delta}(a_ j)^p.
\]
\mbox{}
\\
Bearing in mind \eqref{Eq-SC-0}, this gives

\begin{equation}\label{Eq-SC-3}
\begin{split}
\|E\|_{S_ p}^p &\le C_ 2^p\cdot  2^{-mp}\,\sum_{j}\tau(a_ j)^{\frac{3p}{2}-4} \, \widehat{\mu}_{\delta}(a_ j)^p \,
 \bigg ( \sum_{n} \frac{S_ n (a_ j)}{\tau(a_ n)^{p/2}}\bigg ) \cdot \bigg (\sum_ k S_ k(a_ j)\bigg ).
\end{split}
\end{equation}

\mbox{}
\\
On the other hand,  we have
\begin{equation}\label{Eq-SC-4}
\begin{split}
  \sum_{k} S_ k (a_ j) & = \sum_{k} \int_{D(3\delta\tau(a_j))} |f_{a_k}(z)|^{p/2}\,\omega(z)^{p/4} \,  dA(z)\\
\\
&= \int_{D(3\delta\tau(a_j))} \bigg (\sum_ k |F_{a_ k,N}(z)|^{p/2}\,\,\tau(a_ k)^{-p/2} \bigg )\,\,\omega(z)^{p/4} \,  dA(z).
\end{split}
\end{equation}
\mbox{}
\\
 Now, we claim that
\begin{equation}
\begin{split}\label{eq:Sp2}
\sum_ k \tau(a_ k)^{-p/2}\, |F_{a_ k, N}(z)|^{p/2}\lesssim \,
\tau(z)^{-p/2}\,\omega(z)^{-p/4}\,.
\end{split}\end{equation}
\mbox{}
\\
 In order to prove \eqref{eq:Sp2}, note first that using the
estimate \eqref{BL1} in Lemma \ref{Borichevlemma}, the estimate \eqref{asymptau} and $(iv)$ of Lemma
\ref{Lcoverings}, we deduce that
\begin{equation}
\begin{split}\label{eq:zk}
\sum_{\{a_k\in D(\delta_ 0 \tau(z))\}}\!\! \!\!\!\! \tau(a_
k)^{-p/2}&\,|F_{a_ k, N}(z)|^{p/2}
\\
&\lesssim \omega(z)^{-p/4}\sum_{\{a_k\in D(\delta_ 0
\tau(z))\}}\!\!\!\!\!\! \tau(a_ k)^{-p/2}
 \lesssim \tau(z)^{-p/2}\,\omega(z)^{-p/4}\,.
\end{split}
\end{equation}
\par On the other hand, an application of  \eqref{BL2}
gives
\begin{displaymath}
\begin{split}
\sum_{\{a_k\notin D(\delta_ 0 \tau(z))\}}\!\! \!\!\!\! \tau(a_
k)^{-p/2} & \,|F_{a_ k, N}(z)|^{p/2}
\\ & \lesssim
\omega(z)^{-p/4}\tau(z)^{\frac{3Np}{2}-\frac{p}{2}-2}\sum_{\{a_k\notin D(\delta_ 0
\tau(z))\}}\!\! \frac{\tau(a_ k)^2}{\,\,|z-a_k|^{3Np/2}}
\\
&=\omega(z)^{-p/4}\tau(z)^{\frac{3Np}{2}-\frac{p}{2}-2}\,\sum_{j=0}^{\infty}\,\sum_{a_
k\in R_ j(z)} \frac{\tau(a_ k)^2}{\,\,|z-a_k|^{3Np/2}},
\end{split}
\end{displaymath}
where
\begin{displaymath}
R_ j(z)=\left\{\zeta\in\D: 2^{j}\delta_ 0\tau(z)<|\zeta-z|\leq
2^{j+1}\delta_ 0\tau(z)\right\},\, \, j=0,1,2\dots
\end{displaymath}
\mbox{}
\\
 Now observe that, using condition (B) in the definition of the class $\mathcal{L}$, it is easy to see that,
for $j=0,1,2,\dots$,
$$
D(\delta_0\tau(a_k))\subset
D(5\delta_02^{j}\tau(z))\quad \textrm{ if }\quad a_k\in
D(2^{j+1}\delta_0\tau(z)).$$
This fact together with the finite
multiplicity of the covering (see Lemma \ref{Lcoverings}) gives
\begin{displaymath}
\sum_{a_ k\in R_ j(z)}\!\!\tau(a_ k)^2\lesssim
\,m\Big(D(5\delta_ 0 2^j\tau(z))\Big)\lesssim 2^{2j}\tau(z)^2.
\end{displaymath}
Therefore, as $3Np-4>2p$,
\begin{equation*}
\begin{split}\label{eq:ii}
\sum_{\{a_k\notin D(\delta_ 0 \tau(z))\}}\!\! \!\!\!\! \tau(a_
k)^{-p/2}\,|F_{a_ k, N}(z)|^{p/2}
& \lesssim \omega(z)^{-p/4}\tau(z)^{-\frac{p}{2}-2}\sum_{j=0}^\infty
2^{-\frac{3Npj}{2}}\sum_{a_k\in R_ j(z)}\!\!\tau(a_ k)^2
\\ & \lesssim
\omega(z)^{-p/4}\tau(z)^{-p/2}\sum_{j=0}^\infty 2^{\frac{(4-3Np)}{2}j}
\\ & \lesssim \omega(z)^{-p/4}\tau(z)^{-p/2},
\end{split}\end{equation*}
which together with \eqref{eq:zk}, proves  the claim \eqref{eq:Sp2}.

\mbox{}
\\
Putting \eqref{eq:Sp2} into \eqref{Eq-SC-4} gives
\[
 \sum_{k} S_ k (a_ j) \lesssim \tau(a_ j)^{2-p/2}.
\]
\mbox{}
\\
Similarly,

\[
 \sum_{n} \frac{S_ n (a_ j)}{\tau(a_ n)^{p/2}} \lesssim \tau(a_ j)^{2-p}.
\]
Putting these estimates into \eqref{Eq-SC-3} we finally get

\[
\|E\|^p_{S_ p} \le C_ 2^p\cdot C_ 3 \cdot  2^{-mp}\,\sum_{j} \widehat{\mu}_{\delta}(a_ j)^p
\]
\mbox{}
for some positive constant $C_ 3$. Combining this with  \eqref{DTDE}, \eqref{NDD} and choosing $m$ large enough so that $$C_ 2^p\cdot C_ 3 \cdot  2^{-mp}\le C_ 1/2,$$ then we deduce that

\[
\sum_{j} \widehat{\mu}_{\delta}(a_ j)^p\le \frac{C_ 1}{2} \,\|T\|^p_{S_ p}\le C_ 4 \|T_{\mu}\|^p_{S_ p}.
\]
\mbox{}
\\
Since this holds for each one of the $M$ subsequences of $\{z_ n\}$, we obtain

\begin{equation}\label{E-SC-F}
\sum_{n} \widehat{\mu}_{\delta}(z_ n)^p\le C_ 4\, M \,\|T_{\mu}\|^p_{S_ p}
\end{equation}
for all locally finite positive Borel measures $\mu$ such that
\[
\sum_{n} \widehat{\mu}_{\delta}(z_ n)^p<\infty.
\]
Finally, an easy approximation argument then shows that \eqref{E-SC-F} actually holds for all locally finite positive Borel measures $\mu$. The proof is complete.
\end{proof}


\end{document}